\documentclass{fundam}

\usepackage{amsmath, amscd, amssymb, latexsym}
\usepackage{enumerate}
\usepackage[shortlabels]{enumitem}
\usepackage{stackrel}
\usepackage{pdflscape}
\usepackage{graphicx}
\usepackage{float}
\usepackage{xcolor}

\newcommand{\Q}{\mathbb{Q}}

\newcommand{\R}{\mathbb{R}}

\newcommand{\Z}{\mathbb{Z}}

\newcommand{\nl}{\medskip\noindent}

\newcommand \str [1] {$\langle  {#1} \rangle$}
\newcommand \Str [1] {\text{Str}{#1}}

\newcommand \Strem [1] {\text{\em Str}{#1}}
\newcommand\Ls{$\langle  \R, +,< ,\Z\rangle$}
\newcommand\Ss{$\langle  \R, +,<,1 \rangle$}

\newcommand\Xs[1]{$\langle  \R, +,<,1,#1 \rangle$}
\newcommand\LXs[1]{$\langle  \R, +,<,\Z,  #1 \rangle$}
\newcommand\SXs[1]{$\langle  \R, +,<,1,  #1 \rangle$}
\newcommand{\Sx}{{\mathcal{S}_X}}
\newcommand{\Tx}{{\mathcal{T}_X}}

\usepackage[colorlinks = true, linkcolor = black, urlcolor = black, citecolor = black, anchorcolor = black]{hyperref}

\begin{document}

\setcounter{page}{15}
\publyear{22}
\papernumber{2140}
\volume{188}
\issue{1}

   \finalVersionForARXIV

\title{Decidability of Definability Issues in the Theory of Real Addition}

\author{Alexis B\`{e}s\thanks{Address for correspondence: Univ. Paris Est Creteil, LACL, F-94010 Creteil, France. \newline \newline
                    \vspace*{-6mm}{\scriptsize{Received July 2022; \ accepted   November  2022.}}}
\\
Univ. Paris Est Creteil, LACL, F-94010 Creteil, France \\
bes@u-pec.fr
\and Christian Choffrut\\
IRIF (UMR 8243), CNRS and Universit\'e Paris 7 Denis Diderot, France\\
Christian.Choffrut@irif.fr
}

\maketitle

\runninghead{A. B\`{e}s and Ch. Choffrut}{Decidability of Definability Issues in the Theory of Real Addition}

\begin{abstract}
Given a subset of $X\subseteq \R^{n}$  we can associate  with every point  $x\in \R^{n}$ a vector space $V$ of maximal dimension with the property that for some ball centered at $x$, the subset $X$ coincides inside the ball with a union of lines parallel to $V$. A point is singular if $V$ has dimension $0$.

In an earlier paper we proved that a \Ls-definable relation $X$ is    \Ss-definable  if and only if the number of singular points is finite and every rational section of $X$  is \Ss-definable, where a rational section is a set obtained from $X$ by fixing some component to a rational value.

Here we show that we can dispense with the hypothesis of $X$ being \Ls-definable by requiring that the components of the singular points be rational numbers. This provides a topological characterization of first-order definability in the structure \Ss.
It also allows us to deliver a self-definable criterion (in Muchnik's terminology) of \Ss- and \Ls-definability for a wide class of relations, which turns into an effective criterion provided that the corresponding theory is decidable. In particular these results apply to the class of
so-called $k-$recognizable relations  which are defined by finite Muller automata via the representation of the reals in a integer basis $k$,  and allow us to prove that it is decidable whether a $k-$recognizable relation (of any arity) is $l-$recognizable for every base $l \geq 2$.
\end{abstract}

\section{Introduction}

In his seminal work on Presburger Arithmetic \cite{Muchnik03}, Muchnik provides a  characterization of definability of a relation $X \subseteq \Z^n$ in  $\langle \Z,+ , <\rangle$ in terms of sections of $X$ and local periodicity properties of $X$. It also shows that the characterization can be expressed as a $\langle \Z,+, <,X \rangle$-sentence, and thus can be decided if $\langle \Z,+, <, X \rangle$ is decidable. As an application Muchnik proves that it is decidable whether a $k-$recognizable relation $X \subseteq \Z^n$ is $\langle \Z,+,< \rangle$-definable. Recall that given an integer $k \geq 2$, $X$ is $k-$recognizable if it is recognizable by some finite automaton whose inputs are the base-$k$ encoding of integers (see \cite{BHMV94}).

The present paper continues the line of research started in \cite{BC2020}, which aims to extend Muchnik's results and techniques to the case of reals with addition. Consider the structure  \Ss\     of the additive ordered group of reals along with the
constant $1$.  It is well-known that the subgroup $\Z$ of integers  is not first-order-definable in this structure. Let \Ls\ denote the expansion of \Ss\ with the unary predicate ``$x\in \Z$''. In \cite{BC2020} we prove a topological characterization of \Ss-definable relations in the family of \Ls-definable relations, and use it to derive, on the one hand, that it is decidable whether or not a relation on the reals definable in  \Ls\ can be defined in \Ss\  and on the other hand that there is no intermediate structure between \Ls\ and \Ss\ (since then, the latter result has been generalized by Walsberg \cite{Wal20} to a large class of $o-$minimal structures)

We recall the topological characterization of \Ss\ in \Ls,   \cite[Theorem 6.1]{BC2020}. We say that the neighborhood of a point $x\in \R^{n}$
relative to a relation $X\subseteq \R^{n}$   has  a {\it stratum} if there exists a direction such that the intersection of X with any sufficiently small neighborhood around $x$ is the trace of a union of lines parallel to the given direction. When $X$ is \Ss-definable, all points have strata, except finitely many which we call singular. In \cite{BC2020} we give necessary and sufficient conditions for a \Ls-definable relation $X \subseteq \R^n$ to be \Ss-definable, namely {(FSP)}: it has finitely many singular points  and  {(DS)}: all intersections of $X$ with arbitrary hyperplanes parallel to $n-1$ axes and having rational components on the remaining axis are \Ss-definable.
We asked whether it is possible to remove the assumption that the given relation  is \Ls-definable.
In the present paper we prove that the answer is positive if a new  assumption is added, see below.
Let us first explain the structure of the proof in \cite{BC2020}. The necessity of the two conditions (FSP) and (DS)  is easy. The  difficult part  was their sufficiency and it used very specific properties of the \Ls-definable relations, in particular the fact that   \Ss- and
 \Ls-definable relations are  locally indistinguishible. In order to show the existence of
 a \Ss-formula  for $X$ we showed   two intermediate
 properties,    (RB):  for every nonsingular point $x$, there exists a basis of the strata subspace composed of vectors with rational components, and  (FI): there are finitely many ``neighborhood types", i.e., the equivalence relation $x \sim y$ on $\R^n$ which holds
   if there exists $r>0$ such that ($x+w \in X \leftrightarrow y+w \in X$ for every $|w|<r$) has finite index.

 When passing from the characterization of \Ls-definable relations to the characterization of
 general ones the topological characterization uses  the same intermediate
   properties   but  they are much more delicate to establish and  an   extra condition (RSP) is required:  all singular points of $X$ have rational components.
Moreover we show that this characterization is effective under natural conditions. Indeed,
 if every nonempty  \SXs{X}-definable relation contains a point with rational components, then the \Ss-definability of $X$ is
 expressible in the structure \SXs{X} itself.   The crucial point is the notion of \emph{quasi-singular} points
generalizing that of singular points. We were forced to consider this new notion because the \Xs{X}-predicate which
defines singular points in \Ls\ no longer defines them in  general structures.
In so doing  we can turn
 the criterion for \Ss-definability into an effective criterion provided that the theory of \SXs{X} is decidable. More precisely we show that for every
decidable expansion $\+M$ of \Ss\ such that every nonempty  $\+M$-definable relation contains a point with rational components, one can decide
whether or not a given  $\+M$-definable relation is \Ss-definable.

We extend the result of \Ss-definability of a general relation to that of
 \Ls-definability. Every relation
on the reals can be uniquely decomposed into  some relations on the integers  and some relations on the unit hypercubes (\cite{BFL08}, see also \cite{FV59}). This decomposition yields
a simple characterization of the \Ls-definable relations, which is expressible in \linebreak \LXs{X} provided that all  nonempty \LXs{X}-definable relations
 contain a point with rational components.
Combining the result on \Ss-definability for the reals and Muchnik's result on $\langle  \Z, +,<,1 \rangle$-definable integer relations
 we show that for every decidable expansion $\+N$ of \Ls\ such that every nonempty
$\+N$-definable relation contains a point with rational components, one can decide whether or not a given $\+N$-definable relation is \Ls-definable.

 We also study a particularly significant case.
	The notion of $k$-recognizability for relations on integers can be extended to the case of relations on reals, by considering Muller automata which read infinite words encoding reals written in base $k$, see \cite[Definition 1]{BRW1998}.
 The class of $k-$recognizable
 relations coincides with the class of relations definable in some expansion of  $\langle  \R, +,< ,\Z\rangle$  of the form  $\langle \R,\Z,+,<,X_k \rangle$ where $X_k$ is a base
 dependent ternary predicate \cite[section 3]{BRW1998}.  This expansion satisfies the above required condition since it has a decidable theory and every nonempty
 $k-$recognizable relation contains a point with rational components.
The  \Ls-definable relations define a subclass which has a very specific relevance since it coincides with the class of relations which are
 $k-$recognizable for every $k \geq 2$ \cite{BB2009,BBB2010,BBL09}. A consequence of our result is  that given a $k-$recognizable relation
 it can be decided if it is
 $\ell-$recognizable for all bases $\ell \geq 2$. This falls into the more general issue of  finding effective characterizations of subclasses of $k-$reco\-gnizable relations.
A previous result of this type was  proved by Milchior  in \cite{Milchior17} by showing that it is decidable
whether or not a weakly $k-$recognizable subset of $\R$ is definable in \Ss, where  ``weak'' is
defined as a natural condition on the states of a deterministic automaton.

 We give a short outline of our paper. Section \ref{sec:prelim} gathers basic definitions and notation. In Section \ref{sec:useful} we recall the main useful definitions and results from \cite{BC2020} in order to make the paper selfcontained. In Section \ref{sec:main} we show  that the conjunction of conditions (RSP), (RB) and (FI) characterizes the  \Ss-definable relations. In Section \ref{sec:selfdef}  we deal with the self-definable criterion of \Ss-definability. We  introduce the crucial notion of quasi-singular point and show that it is definable in \SXs{X}. We also provide an alternative, inductive,   formulation of \Ss-definability  for $X$: every relation obtained from $X$
by assigning fixed real values to  arbitrary components contains finitely many quasi-singular points. We then show how to extend the results to the case of \Ls. In Section \ref{sec:applications} we show that the self-definable criterion of \Ss-definability (resp. \Ls-definability) of a relation $X \subseteq \R^n$ can be turned into an effective criterion provided that $X$ is definable in a suitable decidable theory, and apply the result to the class of $k$-recognizable relations.

 \subsubsection*{Other related work.}
 Muchnik's approach, namely expressing in the theory of the structure a property of the structure itself, can be used in other settings.
We refer the interested reader to the discussion in  \cite[Section 4.6]{SSV14} and also to \cite{PW00,Bes13,Milchior17} for examples of such structures.
A similar method has already been used in 1966, see \cite[Thm 2.2.]{GS66} where the authors are able to express in Presburger theory whether or not
a Presburger subset is the Parikh image of a  context-free language.

The theory of (expansions of) dense ordered groups has been studied extensively in model theory, in particular in connection with o-minimality, see e.g. \cite{DG17,DMS10}. Let us also mention a recent series of results by Hieronymi which deal with expansions of \Ls, and in particular with the frontier of decidability for such expansions, see, e.g., \cite{Hie19} and its bibliography.

\section{Preliminaries}
\label{sec:prelim}

Throughout this work we assume the vector space $\R^{n}$ is provided with the    metric
$L_{\infty}$ (i.e., $|x|=\max_{1\leq i\leq n} |x_{i}|$).  Let $B(x,r)$ denote the open ball centered at $x\in \R^{n}$ and of radius $r>0$. Given $x,y \in \R^n$ let  $[x,y]$   (resp. $(x,y)$) denote the closed segment (resp. open segment) with extremities $x,y$. We use also notations such as $[x,y)$ or $(x,y]$ for half-open segments.

Let us specify our logical conventions and notations.  We work within first-order predicate calculus with equality.  We identify  formal symbols and their interpretations.
We are mainly concerned with the structures  \Ss\ and  \Ls.
Given a structure $\cal M$ with domain $D$ and $X \subseteq D^n$, we say that $X$ is {\em definable in $\cal M$}, or {\em $\cal M$-definable}, if there exists a formula $\varphi(x_1,\dots,x_n)$ in the signature of $\cal M$ such that $\varphi(a_1,\dots,a_n)$ holds in $\cal M$ if and only if $(a_1,\dots,a_n) \in X$ (this corresponds to the usual notion of {\em definability without parameters}).

The \Ss-theory  admits quantifier elimination in the following sense, which can be interpreted geometrically
as saying that a \Ss-definable relation is a finite union of closed and open polyhedra.

\begin{theorem}{\cite[Thm 1]{FR75}}
\label{th:quantifier-elimination-for-R-plus}
Every formula in  \Ss\  is equivalent to a  finite Boo\-lean combination  of inequalities  between linear combinations  of variables with coefficients in $\Z$ (or, equivalently, in $\Q$).

In particular every nonempty \Ss-definable relation contains a point with rational components.
\end{theorem}

\section{Local properties of real relations}
\label{sec:useful}

Most of the definitions and results in this section  are taken from \cite{BC2020}.
{
These are variants of notions and results already known in computational geometry, see e.g. \cite{BN88,BBD12} for the case of  \Ss-definable relations.
}
We only give formal  proofs for the new results. In the whole section we fix $n \geq 1$ and  $X \subseteq \R^n$.

\subsection{Strata}\label{subsection:strata}

The following clearly defines  an equivalence relation.

\begin{definition}
\label{de:same-neighborhood}
Given $x,y \in \R^{n}$ we write  $ x \sim_X y$ or simply $ x\sim y$
when $X$ is understood, if there exists a real $r>0$ such that
the translation $w \mapsto w +y-x$ is a one-to-one mapping from $B(x,r)\cap X$ onto $B(y,r)\cap X$.
\end{definition}

\begin{example}
\label{ex:square}
Let $X$ be a closed subset of the plane delimited by a square. There are ten $\sim_X$-equivalence classes:
the set of points interior to the square, the set of points interior to its complement, the four vertices and the
four open edges.
\end{example}

Let  ${\mathcal Cl}(x)$ denote the $\sim$-equivalence class to which $x$ belongs.

\begin{definition}\label{de:strata}\hfill\break

\vspace*{-8mm}
\begin{enumerate}
\item Given a non-zero vector $v \in \R^n$ and a point $y\in \R^n$, let $L_{v}(y)=\{y+\alpha v \ | \ \alpha \in \R\}$ be
the line passing through $y$ in the direction $v$.
More generally, if $X\subseteq \R^n$ let  $L_{v}(X)$ denote the set $\bigcup_{x\in X} L_{v}(x)$.

\item A non-zero vector $v \in \R^n$  is an  $X$-\emph{stratum} at $x$
(or simply a \emph{stratum}  when $X$ is understood)
if there exists a real $r>0$ such that
\begin{equation}
\label{eq:saturation}
B(x, r)  \cap L_{v}(X \cap  B(x, r)  )   \subseteq  X.
\end{equation}
This can be seen as saying that inside the ball $B(x,r)$, the relation $X$ is a union of lines parallel to $v$.
By convention the zero vector is also considered as a stratum.

\item The set of $X$-strata at $x$ is denoted $\text{Str}_{X}(x)$ or simply $\text{Str}(x)$.
\end{enumerate}
\end{definition}

\begin{proposition}\cite[Proposition 3.4]{BC2020}
\label{pr:strata-subspace}
For every  $x\in \R^{n}$ the set $\Strem(x) $ is  a vector subspace of
$\R^{n}$.
\end{proposition}

\begin{definition}
\label{de:dimension}
The \emph{dimension} dim$(x)$ of a point $x \in \R^n$ is the dimension of the subspace $\Str(x)$.
We say that $x$ is a $d$-{\em point} if $d=\dim(x)$.
Moreover if $d=0$ then $x$ is said to be $X$-\emph{singular}, or simply \emph{singular}, and
otherwise it is \emph{nonsingular}.
\end{definition}

\begin{example}\label{ex:square2}(Example \ref{ex:square} continued)  Let $x \in \R^2$. If $x$ belongs to the interior of the square or of its complement, then $\Str(x)= \R^2$. If $x$ is one of the four vertices of the square then
we have $\Str(x)=\{0\}$, i.e., $x$ is singular. Finally, if $x$ belongs to an open edge of the square but is not a
vertex, then $\Str(x)$ has dimension 1, and two  points of  opposite edges have the same strata subspace,
 while two points of  adjacent edges  have different strata subspaces.
\end{example}

 It can be  shown that all strata at  $x$ can be defined with respect to a common value $r$ in expression~(\ref{eq:saturation}).

\begin{proposition}  \cite[Proposition 3.9]{BC2020}
\label{pr:uniform-radius}
For every  $x\in \R^{n}$ there exists a real $r>0$ such that for every  $v\in \Strem(x)\setminus \{0\}$
we have
$$
B(x, r)  \cap  L_{v}(X \cap  B(x, r)) \subseteq X.
$$
\end{proposition}

\begin{definition}
A \emph{$X$-safe radius} (or simply a \emph{safe radius} when $X$ is understood) for $x$ is a real $r>0$ satisfying the condition of Proposition \ref{pr:uniform-radius}.
Clearly if $r$ is safe then so are all $0<s\leq r$. By convention every real is
a safe radius  if   $\Str(x)=\{0\}$.
\end{definition}

\begin{example}(Example \ref{ex:square} continued)  For an element $x$ in the interior of the square
or the interior of its complement a safe radius is  the (minimal) distance from $x$ to the edges of the square.
If $x$ is a  vertex
then $\Str(x)=\{0\}$  and every $r>0$ is safe for $x$. In all other cases $r$ can be chosen as the minimal distance
of $x$ to a vertex.
\end{example}

\begin{remark} \label{re:sim-and-str}
If $x\sim y$ then  $\text{\em Str}(x) =\text{\em Str}(y) $, therefore  given an $\sim$-equivalence class $E$, we may define  $\Str(E)$  as the set of common strata of all $x\in E$.

Observe that the converse is false.
{
	In Example \ref{ex:square} for instance, points in the interior and points in the complement of the interior of the square have the same set of strata, namely $\R^2$, but are not $\sim-$equivalent.
}
\end{remark}

It is possible to  combine the notions of strata and of safe radius.

\begin{lemma}  \cite[Lemma 3.13]{BC2020}
\label{le:strx-subset-stry}
Let  $x\in \R^{n}$  and $r$ be a safe radius for $x$. Then for all $y\in B(x,r)$ we have
$\Strem{(x)}\subseteq \Strem{(y)}$.
\end{lemma}

\begin{example}\label{ex:square3}(Example \ref{ex:square} continued) Consider a point $x$ on an (open) edge
	of the square and a safe radius $r$ for $x$. For every point $y$ in $B(x,r)$ which is not on the edge we have
	$\Str(x)\subsetneq \Str(y)=\R^{2}$. For all other points we have $\Str(x)=\Str(y)$.
\end{example}

Inside a ball whose radius is safe for the center, all points  along a stratum are $\sim$-equivalent.

\begin{lemma}\label{le:tech1}
Let $x$ be non-singular, $v \in \Str(x)\setminus\{0\}$, and $r$ be safe for $x$. For every $z \in B(x,r)$ we have $L_v(z) \cap B(x,r) \subseteq {\mathcal Cl}(z)$.
\end{lemma}

\begin{proof}
Let $z' \in L_v(z) \cap B(x,r) $, and  $s>0$ be such that both $B(z,s),B(z',s)$ are included in  $B(x,r)$. For every $w\in B(0,s)$ we have $z'+w \in L_v(z+w)$ thus $z+w\in X \leftrightarrow z'+w\in X$.
\end{proof}

\subsection{Relativization to affine subspaces}

We relativize the notion of singularity and strata  to an affine subspace $S\subseteq \R^{n}$.
The next definition should come as no surprise.

\begin{definition}
\label{de:H-singular}
Given a subset $X\subseteq  \R^n$, an affine subspace $S\subseteq \R^{n}$  and a point $x\in S$, we say that  a  vector $v \in \R^n \setminus  \{0\}$ parallel to $S$
is an $(X,S)$-\emph{stratum for the point} $x$  if for all sufficiently small $r>0$ it holds
\begin{equation}
\label{eq:relative-stratum}
B(x,r) \cap  L_{v}(X \cap B(x,r) \cap S) \subseteq X.
\end{equation}

By convention the zero vector is also considered as a $(X,S)$-stratum. The set of  $(X,S)$-strata of $x$ is denoted  $\text{Str}_{(X ,S)}(x)$.
 We define the equivalence relation   $x \sim_{(X,S)} y$ on $S$ as follows: $x \sim_{(X,S)} y$ if and only if there exists a real $r>0$ such that $x+w \in X \leftrightarrow y+w \in X$ for every $w \in \R^n$ parallel to $S$ and such that $|w|<r$.
 A point $x\in S$ is $(X,S)$-\emph{singular} if it has no $(X,S)$-stratum. For simplicity when $S$ is the space $\R^{n}$ we maintain  the previous terminology and
speak of   $X$-strata and
$X$-singular points. We say that a real $r>0$ is $(X,S)$-\emph{safe} if (\ref{eq:relative-stratum}) holds for every nonzero $(X,S)-$stratum $v$.
\end{definition}

\begin{remark} Singularity and nonsingularity do not go through restriction to affine subspaces.  E.g., in the real plane, let $X=\{(x,y) \ | \  y<0\}$  and  $S=\{(x,y) \mid x=0\}$. Then the origin is not $X-$singular but it is $(X,S)-$singular. All other elements of $S$ admit $(0,1)$ as an $(X,S)-$stratum thus they are not $(X,S)-$singular.
The opposite situation may occur.  In the real plane, let $X=\{(x,y) \ | \  y<0\} \cup S$.
Then   the origin is  $X-$singular but it is not  $(X,S)-$singular.
\end{remark}

\subsubsection{Relativization of the space of strata}

\begin{lemma}\label{le:general-projstrat}
Let $S$ be an affine hyperplane of $\R^n$ and $x \in S$.  {Let $V$ be the vector subspace generated by $\text{Str}_X(x) \setminus \text{Str}_{(X,S)}(x)$.
If $V\not=\{0\}$ then $\text{Str}_X(x) = V + \text{Str}_{(X,S)}(x)$, and otherwise $\text{Str}_X(x) \subseteq \text{Str}_{(X,S)}(x)$.}
\end{lemma}

\begin{proof}
It is clear that if $V=\{0\}$ then every $X$-stratum of $S$   is  an $(X,S)$-stratum.

Now assume there exists $v\in \text{Str}_X(x) \setminus \text{Str}_{(X,S)}(x)$. It suffices to prove that
 all $w\in \text{Str}_{(X,S)}(x)$ belong to  $\text{Str}_X(x)$. Let $s>0$ be simultaneously $(X ,S)-$safe and $X-$safe for $x$. Let $0<s'<s$ be such that $L_v(z) \cap S \subseteq B(x,s)$ for every $z \in B(x,s')$. Let $y_1,y_2 \in B(x,s')$ be such that $y_1-y_2$ and $w$ are parallel. It suffices to prove the equivalence  $y_1 \in X \leftrightarrow y_2 \in X$ . Let $y'_1$ (resp. $y'_2$) denote the intersection point of $L_v(y_1)$ and $S$ (resp. $L_v(y_2)$ and $S$). We have $y_1,y'_1 \in B(x,s) $, $v \in \text{Str}_X{(x)}$, and $s$ is $X-$safe for $x$, thus $y_1 \in X \leftrightarrow y'_1 \in X$. Similarly we have  $y_2 \in X \leftrightarrow y'_2 \in X$. Now $y'_1,y'_2 \in B(x,s)$,  $y'_1-y'_2$ and $w$ are parallel, and $w \in \text{Str}_{(X,S)}(x)$, which implies  $y'_1 \in X \leftrightarrow y'_2 \in X$ and thus finally
 $y_1 \in X \leftrightarrow y_2 \in X$.
\end{proof}

\begin{corollary}\label{cor:projstrat}
Let $S$ be an  hyperplane of $\R^n$  with underlying vector subspace $V$, and let $x \in S$ be non-singular.
If $\text{Str}_X(x) \setminus  V$ is nonempty then
$
\text{Str}_{(X,S)}(x)=\text{Str}_X(x) \cap V
$.
\end{corollary}

\subsubsection{Relativization of the $\sim$-relation}

\begin{lemma}\label{le:tech2}
Let $S$ be an  hyperplane of $\R^n$, $y,z \in S$, and $v \ne \{0\}$ be a common $X-$stratum of $y,z$ not parallel to $S$. If $y \sim_{(X ,S)} z$ then $y \sim_X z$.
\end{lemma}

\begin{proof}
Assume $y \sim_{(X ,S)} z$, and let $r>0$ be $(X ,S)-$ and $X-$ safe both for $y$ and $z$.
 Since $v$ is not parallel to $S$, there exists  $s>0$ such that  for every $w \in \R^n$ with $|w|<s$, the intersection point of $L_v(y+w)$ (resp.  $L_v(z+w)$) and $S$ exists because $\dim(S)=n-1$ and belongs to $B(y,r)$ (resp. $B(z,r)$).
 It suffices to show that  $y+w \in X \leftrightarrow z+w \in X$. Let   $y+w'$ be the intersection point of $L_v(y+w)$ and $S$.

 By our hypothesis on $s$, $y+w'$ belongs to $B(y,r)$. Moreover  $r$ is $X-$safe for $y$, $v \in \text{Str}_X(y)$, and $w'-w$ is parallel to $v$, therefore $y+w \in X \leftrightarrow y+w' \in X$. Similarly we have $z+w \in X \leftrightarrow z+w' \in X$.
 Now $|w'|<r$, thus by our assumptions $y \sim_{(X ,S)} z$  we have $y+w' \in X \leftrightarrow z+w' \in X$ and therefore
$y+w \in X \leftrightarrow z+w\in X$.
\end{proof}

Next we consider a particular case for $S$ which plays a crucial role in expressing the characterisation stated in the main theorem.
It is also a tool for reasoning by induction in Section \ref{subsec:alternative}.

\begin{definition}
\label{de:section}
Given an index $0\leq i < n$ and a real $c\in \R$ consider the hyperplane
$$
H= \R^{i}\times \{c\} \times \R^{n-i-1}.
$$
The intersection $X \cap H$ is
called a \emph{section} of $X$. It is a \emph{rational section} if $c$ is a rational number.
We define $\pi_{H}: \R^{n} \rightarrow  \R^{n-1}$ as $\pi_{H}(x_1,\dots,x_n)=(x_1,\dots,x_{i-1},x_{i+1},\dots,x_n)$.
\end{definition}
The following facts are easy consequences of the above definitions:
for all $x,y \in H$ and $v $ a vector parallel to $H$ we have:
\begin{enumerate}
\item $x \sim_{(X,H)} y$ if and only if $\pi_{H}(x) \sim_{\pi_{H}(X)} \pi_{H}(y)$
\item $v \in \text{Str}_{(X ,H)}(x)$ if and only if $\pi_{H}(v) \in \text{Str}_{\pi_{H}(X)}(\pi_{H}(x))$. In particular $x$ is $(X,H)-$singular if and only if $\pi_{H}(x)$ is $\pi_{H}(X)-$singular.
\end{enumerate}

\subsection{Intersection of  lines and equivalence classes}\label{subsection:intersectionlines}

In this section we describe the intersection of a $\sim$-class $E$ with a line parallel to some  $v \in \Str(E)$.
It relies on  the notion of adjacency of $\sim$-classes.

\begin{definition}
	Let  $E$  be a nonsingular  $\sim$- class   and let  $v$ be one of its strata.
	A point $x$ is $v-$\emph{adjacent}
	\emph{to}  $E$ if there exists $\epsilon>0$ such that for all $0< \alpha\leq \epsilon$ we have $x+\alpha v\in E$.
\end{definition}

{
\begin{example}(Example \ref{ex:square} continued)
We specify Example  \ref{ex:square} by choosing the square as the unit square with vertices $(0,0),(0,1),(1,0)$ and $(1,1)$. All elements of the bottom open edge of the square belong to the same $\sim$-class $E$. The vector $v=(1,0)$ is a stratum of $E$. The vertex $(0,0)$ is $v-$adjacent to $E$. Similarly every element of $E$ is also $v-$adjacent to $E$. However the vertex $(1,0)$ is not $v-$adjacent to $E$ (but it is $(-v)-$adjacent to $E$).
\end{example}
}

The notion of adjacency is a property of the $\sim$-class.

\begin{lemma}\cite[Lemma 5.2]{BC2020}
\label{le:congruence} %
Let  $F$  be a $\sim$-class.

\begin{enumerate}
\item For all $x,y\in F$, all nonzero vectors  $v$ and   all $\sim$-classes $E$,   $x$ is
 $v$-adjacent to $E $ if and only if  $y$ is  $v$-adjacent to $E$.

\item For each vector $v$ there exists a most one $\sim$-class $E$ such that $F$ is $v$-adjacent to $E$.
\end{enumerate}
\end{lemma}

Consequently, if for some $x\in F$ and some vector $v$, $x$ is $v$-adjacent to $E$ it makes sense to  say that the class $F$ is $v$-adjacent to $E$.

\begin{lemma}\label{le:open-ter}\cite[Corollary 5.6]{BC2020}
Let $x \in \R^n$ be non-singular, $E={\mathcal Cl}(x)$  and let $v \in \Strem(x)\setminus\{0\}$. The set  $L_{v}(x)\cap E$  is a union of disjoint  open segments
(possibly infinite in one or both directions) of   $L_{v}(x)$, i.e.,
of the form $(y-\alpha v , y+ \beta v)$ with $0< \alpha,\beta\leq \infty$ and  $y\in E$.

If $\alpha < \infty$ (resp. $\beta < \infty$) then the point $y-\alpha v$ (resp. $y+ \beta v$)
belongs to a $\sim$-class $F\not=E$ such that
$\text{dim} (F)< \text{dim} (E)$ and  $F$ is $v$-adjacent  (resp. $(-v)-$adjacent)  (or simply \emph{adjacent}
when $v$ is understood) to $E$.
\end{lemma}

\section{Characterizations of \Ss-definable relations}\label{sec:main}

\subsection{Characterization  in \Ls-definable relations}
\label{subsec:properties-of-Ss-and-Ls}

We recall our previous characterization of \Ss-definable  among \Ls-definable relations.

\begin{theorem}
\label{th:CNS}\cite[Theorem 6.1]{BC2020}
Let $n \geq 1$ and let $X \subseteq \R^n$ be \Ls-definable. Then $X$ is \str{\R,+,<,1}-definable if and only if the following
two conditions hold
\begin{description}
\itemsep=0.9pt
\item  {\em (FSP)} There exist only finitely many $X-$singular points.
\item  {\em  (DS)}  Every rational section of $X$ is \Ss-definable.
\end{description}
\end{theorem}

The necessity of  condition  (FSP) is  proved  by Proposition 4.6 of \cite{BC2020}  and that of   (DS)
is trivial since a rational section is the intersection of two \Ss-definable relations.
The proof that conditions (FSP) and (DS) are sufficient uses several properties of \Ls-definable
 relations which are listed in the form of a proposition below.
\begin{proposition}\label{pr:recap}
Let $n \geq 1$ and $X \subseteq \R^n$ be \Ls-definable. The following holds.
\begin{description}
\itemsep=0.9pt
\item  (RSP) The components of the $X$-singular points are rational numbers \cite[Proposition 4.6]{BC2020}.
\item  (FI) The equivalence relation $\sim$ has finite index and thus the number  of different vector
spaces $\Strem(x)$ is finite when $x$ runs over $\R^{n}$ \cite[Corollary 4.5]{BC2020}.
\item  (RB)  For all  nonsingular points $x$, the vector space  $\Str(x)$
has a rational basis in the sense that it  can be generated by a set of  vectors with rational components \cite[Proposition 4.7]{BC2020}.
\end{description}
\end{proposition}

\subsection{Characterization in arbitrary relations}
\label{sec:caract-effectif}

Now we aim to characterize \Ss-definability for an arbitrary relation $X \subseteq \R^n$.
 We prove that the conditions (FSP),(DS),(RSP)  are sufficient, i.e.,  compared to  Theorem \ref{th:CNS} one can remove the condition ``$X$ is \Ls-definable" and add condition (RSP).

\begin{theorem}
\label{th:crit-n}
Let $n \geq 1$ and $X\subseteq \R^{n}$. Then $X$
is \Ss-definable if and only if it satisfies the three  conditions {\em (FSP), (DS), (RSP)}
\begin{description}
\itemsep=0.9pt
\item  {\em (FSP)} It has only finitely many  singular points.
\item  {\em  (DS)}  Every rational section of $X$ is \Ss-definable.
\item  {\em  (RSP)}  Every singular point has rational components.
\end{description}
\end{theorem}

Observe that the  three conditions are needed, as shown by the following relations which are not \Ss-definable.
\begin{itemize}
\itemsep=0.9pt
\item Consider the binary relation $X=\{(x,x) \ | \ x \in \Z\}$. The singular elements of $X$ are precisely the elements of $X$, thus $X$ satisfies (RSP) but not (FSP). It satisfies (DS) because every rational section of $X$ is either empty or equal to the singleton $\{(x,x)\}$ for some $x \in \Z$, thus it is \Ss-definable.
\item The binary  relation $X=\R \times \Z$  has no singular point thus it satisfies (FSP) and (RSP). However it does not satisfy (DS) since, e.g., the rational section $\{0\} \times \Z$ is not \Ss-definable.
\item The unary relation $X=\{\sqrt{2}\}$  admits $\sqrt{2}$ as  unique singular point, thus it satisfies (FSP) but not (RSP). It satisfies (DS) since every rational section of $X$ is empty.
\end{itemize}

Now we prove Theorem \ref{th:crit-n}.

\begin{proof} The necessity of the first two conditions is a direct consequence of Theorem \ref{th:CNS}, that of the third condition is due to
Proposition~\ref{pr:recap}.

\medskip
Now we turn to the  other direction which is the bulk of the proof and
we proceed in two steps. First we show that properties  (FSP),   (DS) and  (RSP) imply properties (RB) and  (FI)
 (Claims \ref{cl:RB} and \ref{cl:FI})
 and then based on these two properties we show that there exists a \Ss-formula defining~$X$.

\begin{claim}
\label{cl:RB}
If $X$ satisfies conditions (FSP),   (DS) and  (RSP) then it satisfies condition (RB).
\end{claim}

\begin{proof}
We prove that for every non-singular point $x \in \R^n$, $\Str(x)$ has a rational basis.   If $n=1$ this   follows from the fact that for every $x \in \R$ the set $\Str(x)$ is either
equal to $\{0\}$  or equal to $\R$, thus we assume  $n \geq 2$.

For every $i\in \{1, \ldots, n\}$ let $H_{i}=\{(x_{1}, \ldots, x_{n})\in \R^{n}\mid  x_{i}=0\}$.
Let us call \emph{rational {$i-$hyperplane}} any hyperplane $S$ of the form $S=\{(x_{1}, \ldots, x_{n})\in \R^{n}\mid  x_{i}=c\}$ where $c \in \Q$.
The underlying vector space  of $S$ is $H_i$.

Let $x$ be a $d-$point with $d\geq 1$, i.e., a point for which   $V=\Str(x)$ has dimension $d$. For $d=n$ the result is obvious. For $1 \leq d <n$ we prove the result by induction on $d$.

\nl \underline{Case $d=1$:} It suffices to show that every $1-$point $x$ has a stratum in $\Q^n$. Let  $v \in \Str(x)\setminus\{0\}$, and let $r>0$ be safe for $x$.  We can find $i \in \{1,\dots,n\}$ and two distinct rational $i-$hyperplanes $S_1$ and $S_2$, not parallel to $v$,  such that $L_v(x)$ intersects $S_1$ (resp. $S_2$) inside $B(x,r)$, say at some point $y_1$ (resp. $y_2$). By Lemma \ref{le:tech1} we have $y_1 \sim x$. By Corollary \ref{cor:projstrat} it follows that
$$\text{Str}_{(X,S_1)}(y_1)=\text{Str}_X(y_1) \cap H_i=\text{Str}_X(x) \cap H_i$$
and the rightmost expression is reduced to $\{0\}$ since $d=1$ and $v \not\in H_i$. This implies that  $y_1$ is $(X,S_1)-$singular, i.e.,  that $\pi_{S_1}(y_1)$ is $\pi_{S_1}(X)-$singular. Similarly $y_2$ is $(X,S_2)-$singular, i.e.,  $\pi_{S_2}(y_2)$ is $\pi_{S_2}(X)-$singular.

\medskip
By condition (DS)  the rational sections $X \cap S_1$ (resp. $X \cap S_2$) are \Ss-definable, thus the   $(n-1)-$ary relations $\pi_{S_1}(X)$ (resp. $\pi_{S_2}(X)$) are also \Ss-definable, and by our hypothesis (RSP) this implies that $\pi_S(y_1)$ (resp. $\pi_S(y_2)$) has rational components. Thus the same holds for $y_1$ and $y_2$, and also for $y_1-y_2$, and the result follows from the fact that $y_1-y_2 \in \text{Str}_{X}(x)$.

\nl\underline{Case $2 \leq d<n$:}
Let $I \subseteq \{1,\dots,n\}$ denote the set of indices $i$ such that $V \not\subseteq H_i$. We have $V \subseteq \bigcap_{i \in \{1,\dots,n\}\setminus I} H_i$ thus $\dim(V) \leq n-(n-|I|)=|I|$, and it follows from our assumption $\dim(V)=d\geq 2$ that $|I|\geq 2$.

\medskip
Now we prove that $V= \sum_{i \in I} (V \cap H_i)$. It suffices to prove $V \subseteq  \sum_{i \in I} (V \cap H_i)$, and this in turn amounts to prove that $\dim(\sum_{i \in I} (V \cap H_i))=d$. For every $1 \leq i \leq n$ we have
$$\dim(V+H_i)=\dim(V)+\dim(H_i)-\dim(V \cap H_i).$$
Now if $i \in I$ then $\dim(V+H_i)>\dim(H_i)$, i.e., $\dim(V+H_i)=n$, which leads to $\dim(V \cap H_i)=d+(n-1)-n=d-1$. Thus, in order to prove $\dim(\sum_{i \in I} (V \cap H_i))=d$ it suffices to show that there exist $i,j \in I$ such that $V \cap H_i \ne V \cap H_j$. Assume for a contradiction that for all $i,j \in I$ we have $V \cap H_i = V \cap H_j$. Then for every $i \in I$ we have
$$V \cap H_i= V \cap \bigcap_{j \in I}{H_j} \subseteq  \bigcap_{j \not\in I}{H_j} \cap \bigcap_{j \in I} H_j=\{0\}$$
which contradicts the fact that $\dim(V \cap H_i)=d-1 \geq 1$.

\medskip
We proved that $V= \sum_{i \in I} (V \cap H_i)$, thus it suffices to prove that for every $i \in I$, $V \cap H_i$ has a rational basis. Let $v$ be an element of  $V \setminus H_i$, and let $r$ be safe for $x$. We can find a rational $i-$hyperplane $S$ not parallel to $v$ and such that the intersection point of  $S$ and $L_v(x)$, say $y$, belongs to $B(x,r)$. By Lemma \ref{le:tech1} (applied to $z=x$) we have
$y \sim x$. Corollary \ref{cor:projstrat} then implies
$$\text{Str}_{(X ,S)}(y)=\text{Str}_X(y) \cap H_i=\text{Str}_X(x) \cap H_i=V \cap H_i$$
which yields
$$
\text{Str}_{\pi_S(X)}(y)=\pi_S(V \cap H_i).
$$
Now by condition (DS), $X \cap S$ is \Ss-definable, and $\pi_S(X)$ as well. Therefore by Proposition \ref{pr:recap} applied to $X \cap S$, the relation $X \cap S$ satisfies (RB) thus $\pi_S(V \cap H_i)$ has a rational basis, and this implies that $V \cap H_i$ also has a rational basis.
\end{proof}

\begin{claim}
\label{cl:FI}
If $X$ satisfies conditions (FSP),   (DS) and  (RSP) then it satisfies condition (FI).
\end{claim}

\begin{proof}
Before proving the claim we need a simple definition.

\begin{definition}
	Given $X \subseteq \R^n$ and a $\sim$-class $E$, we define the \emph{isolated part of $E$} as the subset
	$$Z=\{x \in E \ | \  L_{v}(x)\subseteq E \text{ for all nonzero vectors } v\in \Str(E)\}.$$

	A subset of $\R^n$ is $X$-\emph{isolated} (or simply \emph{isolated} when $X$ is understood) if it is equal to the isolated part of some $\sim$-class.
\end{definition}

{
\begin{example}
	Let $X \subseteq \R^2$ be defined as $X=L_1 \cup L_2$ where $L_1$ denotes the horizontal axis and $L_2$ denotes the open half-line $L_2=\{(x_1,x_2) \ | \ x_2= 1 \text{ and } x_1>0 \}$. In this case there are three $\sim$-classes, namely $E_1=X$, $E_2=\{(0,1)\}$ and $E_3=\R^2\setminus (E_1 \cup E_2)$. Let us describe the isolated part for each of these $\sim$-classes. A point belongs to the isolated part of a  $\sim$-class if whatever stratum is chosen, all points in the direction are trapped in the class. For instance for the $\sim$-class $E_1=X$,  we can show that the isolated part is obtained by deletion of the half-line $L_2$ of $X$, whose points are clearly not trapped. Indeed the subspace $\Str(E_1)$ is generated by the vector $(1,0)$. Therefore for every $v \in \Str(E_1)$, if $x \in L_1$ then $L_v(x)=L_1 \subseteq E_1$, and if $x \in L_2$ then the line $L_v(x)$ intersects $E_2$ thus $L_v(x) \not\subseteq E_1$. This shows the isolated part of $E_1$ is equal to $L_1$.  The $\sim$-class $E_2$ has dimension $0$ thus obviously it is equal to its isolated part. Finally the isolated part of $E_3$ is empty since the vector $v=(0,1)$ is a stratum of $E_3$ and for every $x \in E_3$ the line $L_{v}(x)$ intersects $E_1$.
\end{example}
}

 \begin{lemma}
\label{le:isolated-classes}
 Let $X\subseteq \R^n$ satisfy (FSP),  (DS) and (RSP). We have
 \begin{enumerate}
 \item Let $E$  be a $\sim$-class and $Z$ be its isolated part.
Then $Z$  is a {finite} union of affine subspaces with underlying vector subspace $\Str(E)$ each containing a point with rational components.

 \item There exist finitely many  isolated subsets.
 \end{enumerate}
\end{lemma}

\begin{proof}
By induction on $n$. For $n=1$ if $X$ is equal to $\R$ or to the empty set,
the only isolated set is $X$ and it obviously satisfies $(1)$.  Otherwise a nonempty isolated set $Z$
consists of equivalent points of a $\sim$-class of dimension $0$, i.e., it  is a union of singular points. Now  by (FSP) and (DS) there exist finitely many such points and they have rational components, which  implies  $(1)$ and $(2)$.

Now let $n \geq 1$. All isolated sets $Z$ included in a $\sim$-class $E$ of dimension $0$  satisfy $(1)$, and moreover there are finitely many such sets $Z$. Thus  it suffices to consider the case where $Z \ne \emptyset$ and $\Str(E) \ne \{0\}$.

Let $v \in \Str(E) \setminus\{0\}$ and let $i\in \{1,\ldots, n\}$ be such that $v \not\in H_i$. For every $z \in Z$ we have $L_v(z) \subseteq Z$, thus $Z$ intersects the hyperplane $H_{i}$.
All elements of $Z \cap H_i$ are $\sim_X$-equivalent thus they are also $\sim_{(X,H_i)}$-equivalent.
 Furthermore  for every $x\in Z\cap H_{i}$ we have $\text{Str}_{(X,H_i)}(x)=\text{Str}_{X}(x)\cap H_{i}$ by Corollary \ref{cor:projstrat}
 and thus for every $w\in \text{Str}_{X}(x)\cap H_{i}$ we have $L_w(x) \subseteq Z\cap H_{i}$. This shows that $\pi_{H_i}(x)$ belongs to a $\pi_{H_i}(X)-$isolated set, hence $\pi_{H_i}(Z)$ is included in a
$\pi_{H_i}(X)-$isolated set,  say $W\subseteq \pi_{H_i}(H_i)$.

Now by condition (DS) the set $\pi_{H_i}(X)$ is \Ss-definable, thus by Theorem \ref{th:CNS} it satisfies also (FSP). By our induction hypothesis it follows that $W$ can be written as $W=\bigcup^{p}_{j=1} W_{j}$, where  either all $W_{j}$'s are  parallel affine subspaces with underlying vector space  $\pi_{H_i}(\Str(E))$ each containing some point with rational components (by $(1)$), or each $W_j$ is reduced to a point with rational components (by $(1)$).
Every $W_j$ which intersects $\pi_{H_i}(Z)$ satisfies $W_j \subseteq \pi_{H_i}(Z)$, which shows that $\pi_{H_i}(Z)=  \bigcup_{j \in J} W_{j}$ for some $J \subseteq \{1,\dots,p\}$. That is, we have $Z \cap H_i= \bigcup_{j \in J} W'_{j}$ where each $W'_j=\pi^{-1}_{H_i}(W_j)$. Observe that if $x$ belongs to $W_{j}$ and has rational components then the point $x'=\pi^{-1}_{H_i}(x)$ also has rational components.
Now $Z=(Z \cap H_i)+\Str(E)$ thus $Z= \bigcup_{j \in J} (W'_{j}+\Str(E))$. Since the  underlying vector space of each $W'_j$ is included in $\Str(E)$, this  proves $(1)$.

Concerning $(2)$ we observe that $Z$ is completely determined by $Z \cap H_i$, i.e., $\pi_{H_i}(Z)$. By our induction hypothesis there are finitely many $\pi_{H_i}(X)-$isolated parts $W=\bigcup^{p}_{j=1} W_{j}$   and  each $X$-isolated part is determined by a subset of the form
$\bigcup_{j \in J} W_{j}$ for some $J \subseteq \{1,\dots,p\}$. This proves point $(2)$.
\end{proof}

Now we turn to the proof of Claim \ref{cl:FI}.
 Lemma \ref{le:isolated-classes}  shows that the number of $\sim$-classes having a nonempty isolated part is finite.
It thus suffices to prove that for every $0 \leq d \leq n$ there exist finitely many $d$-classes $E$ having an empty isolated part.

 For $d=0$ the result follows from $(FSP)$ and the fact that each $0-$class is a union of singular points. For $d=n$ there exist at most two $d-$classes, which correspond to elements in the interior of $X$ or the interior of its complement.

For $0 \leq d <n$ we reason by induction on $d$. Observe first that if a $\sim$-class $E$ has dimension $d$ and has an empty isolated part, then there exist $x \in E$ and $v \in \Str(E)\setminus\{0\}$ such that $L_v(x) \not\subseteq E$. By Lemma \ref{le:open-ter} this implies that there exist $y \in L_v(x)$ and a   $\sim$-class $F$ such that  $y\in F$,  $F$ is adjacent to $E$, $\dim(F)<\dim(E)$, and  $[x,y)\subseteq {\mathcal Cl}(x)$.  Now by our induction hypothesis there exist finitely many $\sim$-classes with dimension less than $d$. Thus in order to prove the claim, it suffices to show that there are finitely many $d$-classes to which some $d'$-class with  $d'<d$ is adjacent.

In order to meet a contradiction, assume that there exists a  $d'-$class $F$  which is adjacent to infinitely many $d$-classes, say $E_{j}$ with $j\in J$. We may furthermore assume that for each class $E_{j}$ there is no integer $d'<d''<d$ such that some $d''$-class is adjacent to  $E_{j}$.
Because of Lemma \ref{le:congruence} it is enough to fix an element $y$ in $F$ and investigate the classes to which it is adjacent.

\nl We first consider the case $d'=0$.

 Because of condition (FSP), for some real $s>0$ the point $y$ is the unique singular point in $B(y,s)$. Moreover for every $j \in J$, $F$ is adjacent to $E_j$, thus there exists a point $x_{j}\in E_{j}$ such that $[x_j,y) \subseteq E_j$. Let $\mathit{HL}_{j}$ denote the open halfline with endpoint $y$ and containing $x_j$.
Observe that we necessarily have $\mathit{HL}_{j}\cap B(y,s)\subseteq {\mathcal Cl}(x_{j})$. Indeed, by Lemma \ref{le:open-ter} the condition
$\mathit{HL}_{j}\cap B(y,s)\subsetneq {\mathcal Cl}(x_{j})$  implies that  there exists a point   $z=y + \alpha (x_{j}-y) \in B(y,s)$ such that $\alpha >1$ and $\text{dim}(z)< d$. Since $y$ is the unique singular point in $B(y,s)$ this implies $\text{dim}(z)>0$ but then because of $[x_{j},z)\subseteq  {\mathcal Cl}(x_{j})$  the maximality  condition stipulated for $d'$ is violated.

Let $z_{j}$ be the point on  $\mathit{HL}_{j}$ at distance $\frac{s}{2}$ from $y$ and let $z$ be adherent to the set $\{z_j \ | \ j \in J\}$. The
point  $z$ is nonsingular since $y$ is the unique singular point in the ball $B(y,s)$. Let $v \in \Str(z)\setminus\{0\}$.
Consider
 some $\ell \in \{1, \ldots, n\}$, some  rational  $\ell -$hyperplane $S$ such that $z \not\in S$ and some real $0<t<\frac{s}{2}$ such that $L_{v}(B(z,t))\cap S\subseteq B(z,\frac{s}{2})$. The ball $B(z,t)$ contains infinitely many non $\sim$-equivalent points, and by Lemma   \ref{le:tech2}  their projections on $S$ in the direction $v$ are
 non $\sim_{(X,S)}$-equivalent.
  But by condition  (DS) the relation $X\cap S$ is \Ss-definable, thus $\pi_S(X)$ satisfies condition (FI) of Proposition \ref{pr:recap}, a contradiction.

\nl  Now we consider the case where $d'>0$.

Choose some  $v \in \Str(y)$ and let $r$ be a safe radius for $y$. We can find $0<s<r$, $k \in \{1,\dots,n\}$ and some $k-$hyperplane $S$ not parallel to $v$ such that
$L_v(B(y,s))\cap S \subseteq B(y,r)$.
By definition of $y$, $B(y,s)$ intersects infinitely many pairwise distinct $d-$classes. Given two non $\sim$-equivalent $d$-points $z_1,z_2 \in B(y,s)$, and   their respective projections $w_1,w_2$ over $S$  along the direction $v$, we have $w_1 \not\sim_{(X,S)} w_2$ by Lemma \ref{le:tech2}.
This implies that there exist infinitely many $\sim_{(X,S)}$-classes. However by condition  (DS), the relation $X \cap S$ is \Ss-definable,
 thus $\pi_S(X)$ satisfies condition (FI) of Proposition \ref{pr:recap}, a contradiction.
\end{proof}

Now we turn to the proof of Theorem  \ref{th:crit-n}.
Observe that $X$ is equal to the union of $\sim$-classes of its elements, thus by Claim \ref{cl:FI}, in order to prove that $X$ is \Ss-definable it suffices to prove that all $\sim_X$-classes are \Ss-definable.
More precisely, we prove that each $\sim$-class $E$ is definable from $\sim$-classes $F$ with smaller dimension, i.e., that $E$ is definable in the expansion of \Ss\ obtained by adding a predicate for each such $F$. We proceed by induction on the dimension $d$ of $\Str(E)$.

If $d=0$ then $E$ is a union of singular points, and by (FSP) and (RSP) it follows that $E$ is a finite subset of $\Q^n$ thus is \Ss-definable.

Assume now $0<d\leq n$. By Claim \ref{cl:RB} there exists a rational basis $V(E)=\{v_1,\dots,v_d\}$ of $\Str(E)$. Let $Z \subseteq E$ be the isolated part of $E$ and let $E'=E \setminus Z$. By Lemma \ref{le:isolated-classes} $(1)$, $Z$ is a {finite} union of parallel affine subspaces with  underlying vector space $V(E)$ each containing a point with rational components, thus $Z$ is \Ss-definable.
It remains to prove that $E'$ is \Ss-definable. We use the following characterization of $E'$.

\begin{lemma}\label{le:non-isolated}
For every $x \in \R^n$, we have $x \in E'$ if and only if there exist $1 \leq p \leq d$ and a sequence of pairwise distinct elements $x_0,\dots,x_p \in \R^n$ such that $x_0=x$ and
\begin{enumerate}
\item for every $0\leq k \leq p-1$, $x_{k+1}-x_k\in V(E)$  and $[x_k,x_{k+1})$ does not intersect any $\sim$-class of strictly smaller dimension
than $\dim(E)$
\item if  $F= {\mathcal Cl}(x_{p})$ then $F$ is $(x_{p-1}-x_{p})$-adjacent to $E$ and  $\dim(F)<\dim(E)$.
\end{enumerate}
\end{lemma}

\begin{proof}
We first prove that the conditions are sufficient.  We prove by backward induction  that $[x_k,x_{k+1}) \subseteq E$ for every $0\leq k \leq p-1$. This will imply that $x=x_0 \in E$, and the fact that $x_p-x$ belongs to $\Str(E)$ and $\dim(F)<\dim(E)$ will lead to $x \in E'$. Set $k=p-1$.
By Point 2 of Lemma \ref{le:congruence}   the element   $x_{p}$ is  $(x_{p-1}- x_{p})$-adjacent to $E$, thus $[x_{p-1},x_{p})$  intersects
$E$. Moreover $[x_{p-1},x_{p})$ does not intersect any $\sim$-class $G$ such that $\dim(G)<\dim(E)$, thus by Lemma  \ref{le:open-ter} we have $[x_{p-1},x_{p}) \subseteq E$. For $0\leq k<p-1$, by our induction hypothesis we have $x_{k+1} \in E$. Moreover $[x_{k},x_{k+1})$ does not intersect any $\sim$-class $G$ such that $\dim(G)<\dim(E)$, thus $[x_{k},x_{k+1}) \subseteq E$  again by Lemma  \ref{le:open-ter}.

We prove the necessity. By definition of $E'$ and Lemma \ref{le:open-ter}  there exist $v\in \Str{(E)}$  and $y\in L_v(x)$ such that
 $[x,y)\subseteq E$ and $y\not\in E$. Decompose $v=\alpha_{1} v_{i_{1}} + \cdots + \alpha_{p} v_{i_{p}}$ where
 $0<i_{1}< \cdots <   i_{p}\leq d$ and $\alpha_{1} , \cdots,   \alpha_{p} \not= 0$. We can assume
  w.l.o.g that $y$ is chosen such that $p$ is
 minimal and furthermore that $\alpha_{p}$ is minimal too.
  For $0 \leq k<p$ set $x_{k}=x + \alpha_{1}  v_{i_{1}} + \cdots + \alpha_{k}v_{i_{k}}$. By minimality of $p$ and $\alpha_p$, the segments
 $[x_{0}, x_{1}), \ldots, [x_{p-1}, x_{p})$ intersect no class of dimension less than $\dim(E)$.
 Then $y=x_{p}$ is $(x_{p-1}-x_{p})$-adjacent to $E$.
  \end{proof}

In order to prove that $E'$ is \Ss-definable it suffices to show that we can express in \Ss\ the existence of a sequence $x_0,\dots,x_p \in \R^n$ which satisfies both conditions of Lemma \ref{le:non-isolated}. Observe that  $V(E)$ is finite  and each of its element is \Ss-definable, thus we can express in \Ss\ the fact that a segment is parallel to some element of $V(E)$. Moreover by (FI) there exist finitely many  $\sim$-classes $F$ such that $\dim(F)<\dim(E)$, and all such classes are \Ss-definable by our induction hypothesis. This allows us to express condition $(1)$ in \Ss. For $(2)$ we use again the fact that there are only finitely many classes $F$ to consider and that all of them are \Ss-definable.
\end{proof}

\subsection{An alternative noneffective formulation.}
 \label{subsec:alternative}

  In this section we re-formulate Theorem \ref{th:crit-n} in terms of (generalized) projections of $X$
  by building on  the notion of generalized section which extends that of
section, in the sense that it  allows us  to fix several components.

\begin{definition}
\label{de:generalized-section} Given $n \geq 1$ and $X \subseteq \R^n$,  a \textit{generalized section of $X$} is a relation of the form
\begin{equation}
\label{eq:s-a}
X_{s,a} =\{(x_{1}, \ldots, x_{n})\in X \mid  x_{s_{1}} =a_{{1}}, \ldots, x_{s_{r}} =a_{{r}}\}
\end{equation}
where $r>0$, $s= (s_1,\dots,s_r)$ is an $r-$tuple of integers with $1\leq s_{1} < \cdots <s_{r}\leq n$, and  $a=(a_{{1}}. \ldots, a_{{r}})$ is an $r-$tuple of reals.
When $r=0$ we define $X_{s,a} =X$ by convention, i.e.,  $X$ is a generalized section of itself. If $r>0$ then the section is said to be {\em proper}.  If all elements of $a$ are rational numbers then $X_{s,a}$ is called a {\em rational generalized section of $X$}.

\medskip
In the above definition, each $X_{s,a}$ is a subset of $\R^n$. If we remove the $r$ fixed components $x_{s_{1}},\ldots, x_{s_{r}}$ we can see $X_{s,a}$ as a subset of $\R^{n-r}$, which will be called a {\em generalized projection} of $X$ (resp. a {\em rational generalized projection} of $X$ if   $X_{s,a}$ is a rational generalized section of $X$).
\end{definition}

\begin{proposition}
\label{pr:recurs-cri}
For every $n \geq 1$,  a  relation $X\subseteq \R^{n}$  is \Ss-definable if and only if every {rational generalized
projection} of $X$ has finitely many singular points and these points have rational components.
\end{proposition}

\begin{proof}
The proof goes by induction on $n$. The case $n=1$ is obvious. Assume now $n >1$.

Let $X$ be \Ss-definable and let $Y$ be a rational generalized  projection of $X$. If $Y=X$ then the result follows from Theorem \ref{th:crit-n}. If $Y$ is proper then $Y$ is definable in \str{\R,+,<,1,X} thus it is also \Ss-definable, and the result follows from our induction hypothesis.

Conversely assume that every {rational generalized  projection} of $X$ has finitely many singular points and they have rational components. We show that $X$ satisfies all three conditions of Theorem \ref{th:crit-n}. Conditions (FSP) and (RSP) follow from our hypothesis and the fact that $X$ is a {rational generalized projection} of itself. It remains to prove condition (DS) namely that every rational section of $X$ is \Ss-definable. This amounts to proving that  every rational projection $Z$ of $X$ is \Ss-definable. Now every generalized projection $Y$ of $Z$ is also a  generalized projection of $X$, thus by our induction hypothesis $Y$ has finitely many singular points and they have rational components. Since $Z$ is a proper projection of $X$, by our induction hypothesis it follows that $Z$ is \Ss-definable.
\end{proof}

\section{A definable criterion for \Ss-definability in suitable structures}
\label{sec:selfdef}

 In this section we prove that for every $n \geq 1$ and $X \subseteq \R^n$, if every nonempty \SXs{X}-definable relation contains a point with rational
 components then we can state a variant of Proposition \ref{pr:recurs-cri} which is
  expressible in \SXs{X}. This means that there exists a \SXs{X}-sentence
 (uniform in $X$) which expresses  the fact that $X$ is \Ss-definable. This provides a \emph{definable} criterion for \Ss-definability, similar to Muchnik's result
 \cite{Muchnik03} for definability in Presburger Arithmetic. We also extend these ideas to the case of \Ls-definability.

In this section $\Sx$ stands for the structure \SXs{X}.

\subsection{Quasi-singular points}

We aim to express \Ss-definability of a relation $X \subseteq \R^n$ in the structure $\Sx$ itself. A natural approach is to express the conditions of Proposition \ref{pr:recurs-cri} as an $\Sx$-sentence, however the formulation   involves the set $\Q$ as well as the set of $X$-singular elements. On the one hand $\Q$ is not necessarily $\Sx$-definable,  and on the other hand the naive definition of $X-$singularity involves
the  operation of multiplication  which
is  also not necessarily $\Sx$-definable. For the special case where $X$ is \Ls-definable, we introduced in \cite[Lemma 4.9]{BC2020}
an ad hoc  definition. Yet this definition
does not necessarily  hold when the relation $X$ is no longer assumed to be \Ls-definable. In order to overcome
this difficulty we introduce a weaker property   but
that is still definable in $\Sx$. This proves to be sufficient to establish our result.

\begin{definition}  Let  $X \subseteq \R^n$, $x \in \R^n$, and $r,s$ be two reals such that $0<s<r$.
	\begin{itemize}
		\item a vector $v\in \R^n$  is an $(r,s)$-\emph{quasi-stratum} of $x$ if $|v|\leq s$ and $(y\in X \leftrightarrow y+v \in X)$
		holds for all $y \in \R^n$  such that $y,y+v\in B(x,r)$.
		\item
		We say that $x \in \R^n$ is $X-$\emph{quasi-singular} if it  does not  satisfy the following property:
		\begin{equation}
			\begin{array}{l}
				\label{eq:qs}
				\text{there exist reals } r, s>0 \text{ such that the set of  $(r,s)$-{quasi-strata} of $x$ is nonempty, closed, } \\
				\text{and is stable under } v \mapsto v/2. \vspace*{-6mm}
			\end{array}
		\end{equation}		
	\end{itemize}
\end{definition}\smallskip

It is not difficult to check that if $x$ is not singular and $r$ is safe, then every stratum of $x$ is an $(r,s)$-quasi-stratum for $0<s<r$. However even for $r$ safe, there may exist $(r,s)$-quasi-strata of $x$ which are not strata of $x$, as shown in the following example.

\begin{example}
Let $n=2$, $X=(\Z  \cup \{-{3 \over 2},-{1 \over 2},{1 \over 2},{3 \over 2}\})   \times \R$, and $x=(0,0)$. Then $\Strem(x)$ is generated by the vector $(0,1)$, and every real $r>0$ is safe for $x$. Given $0<s<r$, the $(r,s)$-quasi-strata of $x$ can be characterized as follows:
\begin{itemize}
\itemsep=0.9pt
	\item if $r>{5 \over 2}$ then the $(r,s)$-quasi-strata of $x$ are vectors of the form $(0,l)$ with $|l| \leq s$ (i.e these are the strata of $x$ with norm at most $s$).
	\item if $r \leq {5 \over 2}$ then these are vectors of the form $({k \over 2},l)$ where $k \in \Z$, $k \leq 2r$, and $|({k \over 2},\ell)|\leq s$. Note that for $s<{1 \over 2}$, these are exactly the strata of $x$ with norm at most $s$.
\end{itemize}
\end{example}
	
\begin{lemma}\label{le:defquasi}
Let $X \subseteq \R^n$. The set of quasi-singular elements of $X$ is $\Sx$-definable. Moreover the property
``$X$ has finitely many quasi-singular elements" is  $\Sx$-definable  (uniformly in $X$) .
\end{lemma}

\begin{proof}
The property that $v$ is an $(r,s)$-stratum of $x$ can be expressed by the formula
		$$
		\phi(X,x, r ,s, v)\equiv  0< |v|< s \wedge  \forall y \   (y,y+v\in B(x,r) \rightarrow (y\in X \leftrightarrow y+v \in X)).
		$$
	
The set of $X$-quasi-singular elements can be defined by the formula
\begin{equation}
\label{eq:asterisk}
\begin{array}{ll}
	QS(x,X) \equiv & \neg \exists r \exists s \ ((0<s<r) \wedge  \exists v (\phi(X,x,r,s,v)) \wedge \\
& \forall v\ (\phi(X,x,r,s,v)\rightarrow \phi(X,x,r,s,\frac{v}{2})) \wedge \\
& \forall v \ (((|v|<s) \wedge \forall \epsilon>0 (\exists u \ (\phi(X,x,r,s,u)\wedge |v-u|<\epsilon ))) \rightarrow \phi(X,x,r,s,v)))
\end{array}
\end{equation}

The finiteness of the set of  quasi-singular points can be expressed by the formula
\begin{equation}
\label{eq:finiteQS}
\begin{array}{ll}
FS(X)\equiv &  (\exists t>0 \ \forall x \    (QS(x,X)   \rightarrow |x|<t))\\
&  \wedge (\exists  u>0
(\forall x \forall y ( (QS(x,X)  \wedge QS(y,X)  \wedge x\not=y)\rightarrow |x-y|>u)))
\end{array}
\end{equation}

\vspace*{-5mm}
\end{proof}

\begin{lemma}\label{le:as}
Let  $X \subseteq \R^n$.	 If $x$ is not quasi-singular then for some reals $0<s<r$ there exists an $(r,s)$-quasi-stratum of $x$ and	every $(r,s)$-quasi-stratum of $x$ is a stratum of $x$.
\end{lemma}

\begin{proof}
 In this proof ``quasi-stratum'' stands for ``$(r,s)$-quasi-stratum''.  We consider the negation of $QS$. The matrix of the formula consists of four conjuncts.
 The second  conjunct  asserts that there exists a quasi-stratum. The third conjunct asserts that if a vector $v$ is a quasi-stratum
 then for all integers $p\leq 0$ the vector  $2^{p} v $  is a quasi-stratum.  Also if  $p\geq 0$  and  $|2^{p} v|<s $ then $2^{p} v $
  is a quasi-stratum. Indeed, because $B(x,r)$ is convex, if $y$ and $y+2v$ belong to $B(x,r)$
  then $y+v$ belongs to $B(x,r)$ and we have
\begin{equation}
\label{eq:transitive}
y\in X \leftrightarrow  y+   v\in X \leftrightarrow   z= y+  2 v\in X.
\end{equation}
This  generalizes to any $p\geq 0$ provided  $|2^{p} v|<s$.

\medskip
  We will show that if $v$ is quasi-stratum, then it is a stratum, i.e.,
  if $y,z\in B(x,r)$ and $z\in L_{v}(y)$ then $y\in X\leftrightarrow z \in X$. To fix ideas
  set $z=y + 2^{\ell} \beta v$ for some real $0<\beta<1$. Let $\alpha_{q}= \sum_{-q< i<0}  a_{i} 2^{i}$ with $a_{i}\in \{0,1\}$, be a sequence of dyadic rationals converging to $\beta$.
Since  $|2^{i} v|<s$ holds for all $-q< i<0$, every $2^{i} v $ is a quasi-stratum and therefore so is $\alpha_{q} v$.
Arguing as in (\ref{eq:transitive}) we see that for all $t, t+\alpha_{q} v\in B(x,r)$ we have
$t\in X \leftrightarrow  t+  \alpha_{q} v \in X$. Because $\alpha_{q} <1$ this shows that all $ \alpha_{q} v$
are quasi-strata. The last conjunct  implies that $ \beta v$ is a quasi-stratum
and again using the same argument as in (\ref{eq:transitive}) we get
$$
y\in X \leftrightarrow  y+  \beta v\in X \leftrightarrow    y+  2 \beta v \leftrightarrow \cdots   \leftrightarrow z= y+  2^{\ell} \beta v\in X
$$

\vspace*{-8mm}
\end{proof}

\begin{lemma}\label{le:sing-quasi-sing}
For every $X \subseteq \R^n$, every $X-$singular element is $X-$quasi-singular.
\end{lemma}

\begin{proof} If $x \in \R^n$ is not $X-$quasi-singular then there exist $0<s<r$ and an $(r,s)$-quasi-stratum,
which is a stratum by Lemma \ref{le:as}.
\end{proof}

\begin{lemma}\label{le:equiv-quasi-sing-and-sing-for-Ss}
	If   $X \subseteq \R^n$  is \Ss-definable, every $X-$quasi-singular element is $X-$singular.
\end{lemma}

\begin{proof}
	We prove that if $x \in \R^n$ is not $X-$singular then it is not quasi-singular,  i.e., that it satisfies  $\neg QS(x,X)$,  cf. Expression (\ref{eq:asterisk}). We find suitable values of $r,s$ such that the set of $(r,s)$-quasi-strata of $x$ coincides with the set of strata $v$ of $x$ such that $|v|\leq s$. The result will then follow from the fact that  $\Str(x)$ is a non-trivial vector subspace.
	
	{ The relation $X$ is \Ss-definable thus by  \cite[Corollary 4.4]{BC2020} there exists $r>0$ such that inside the ball $B(x,r)$, $X$ coincides with a finite collection of cones.
	By cone we mean an intersection of open or closed halfspaces delimited by finitely many, say $k$,
	hyperplanes  of dimension $n-1$ and containing $x$. Without loss of generality we can assume that $r$ is safe. We show that a suitable value for $s$ is $s=\frac{1}{k}r$.
	
	Since $r$ is safe}, every stratum $v$ of $x$ with $|v|\leq s$ is also an $(r,s)$-quasi-stratum of $x$. Conversely let $v$ be an $(r,s)$-quasi-stratum of $x$, and assume for a contradiction that $v \not\in \Str(x)$. Then there exists a point $y\in B(x,r)$ such that the line $L_{v}(y)$  intersects $X$ and its complement
	inside $B(x,r)$.
	{
	Let $h$ be any homothety with ratio $0<\lambda \leq 1$ centered at $x$ such that
	the segment  $L_{v}(h(y))\cap B(x,r)$ has length greater than $r$.
	}
	Then, within $B(x,r)$, the line  $L_{v}(h(y))$ decomposes into $2\leq p\leq k$ segments
	which are alternatively inside and outside $X$. One of these segments has length at least $\frac{1}{p}r\geq s \geq |v|$.
	We obtain that for some $z\in L_{v}(h(y))$  we have	$z, z+v  \in B(x,r)$ and $z\in X \leftrightarrow z+ v \not\in X $,
	which contradicts our assumption that $v$ is an $(r,s)$-quasi-stratum of $x$.
\end{proof}

Note that in Lemma \ref{le:equiv-quasi-sing-and-sing-for-Ss} the condition that $X$ is \Ss-definable cannot be removed. Consider, e.g., $X=\R \times \Q$. Then it can be shown that for all $x \in \R^2$ we cannot find any reals  $0<s<r$ for which the set of $(r,s)$-quasi-strata of $x$ is closed,
and this implies that $x$ is $X-$quasi-singular. However $x$ is not $X-$singular since $(1,0)$ is an $X-$stratum for $x$.

\subsection{Alternative characterization of \Ss-definability in $\Sx$.}

 We can state the following variant of Theorem \ref{th:crit-n} for $\Sx$-definable relations under the hypothesis that
 all nonempty  $\Sx$-definable relations $Y$ contain a point with rational components
 (recall that  $\Sx$ stands for
the structure \SXs{X} where $X$ is some fixed but arbitrary relation). Observe that this implies that all
 definable finite subsets $Y$ of $\R^{n}$ are included in $\Q^{n}$. Indeed, for all points $y\in Y$ with rational components
 the set $Y\setminus \{y\}$ is $\Sx$-definable.

\begin{proposition}
\label{pr:altern-cri-quasi} Let $n \geq 1$ and  $X\subseteq \R^{n}$ be such that  every nonempty $\Sx$-definable relation contains a point with rational components. Then $X$ is \Ss-definable if and only if every {generalized projection} of $X$ has finitely many quasi-singular points.
\end{proposition}

\begin{proof}
We proceed by induction on $n$.

\smallskip
\underline{Case $n=1$.} Assume first that $X$ is \Ss-definable. The only  {generalized
projection} of $X$ that need to be studied is  $X$ itself. Now by Theorem \ref{th:crit-n}, $X$ has finitely many singular points, and by Lemma \ref{le:equiv-quasi-sing-and-sing-for-Ss} these are precisely its  quasi-singular points.

Conversely assume that every generalized projection of $X$ has finitely many quasi-singular points. If the generalized projection is not $X$, then it is a singleton and there is nothing to check. It remains to consider the case where the projection is equal to $X$. By
Lemma \ref{le:sing-quasi-sing} this implies that $X$ has finitely many singular points. Now $X \subseteq \R$ thus the set of $X-$singular points
coincides with the topological boundary $Bd(X)$ of $X$. It follows that $Bd(X)$ is finite, i.e., $X$ is the union of finitely many intervals. Moreover
$Bd(X)$ is $\Sx$-definable
and by our assumption on $\Sx$ it follows that $Bd(X) \subseteq \Q$ thus every $X-$singular point is rational. The result follows from Theorem \ref{th:crit-n}.

\smallskip
 \underline{Case $n>1$.}
Assume first that $X$ is \Ss-definable. The relation satisfies  property (FSP) of Theorem  \ref{th:crit-n} and  by  Lemma \ref{le:equiv-quasi-sing-and-sing-for-Ss} it has finitely many quasi-singular points. It thus suffices to consider proper subsets of $X$.
Assume without loss of generality that the projections are obtained by freezing the $0<p\leq n$ first components.  For every $a= (a_1,\dots,a_p)\in \R^{p}$, consider the projection
 $$
 X_{a}= \{(x_{p+1},\dots,x_n) \in \R^{n-p} \mid  (a_1,\dots,a_p,x_{p+1},\dots,x_n) \in X\}.
 $$
Consider the set $A$ of elements $a=(a_1,\dots,a_p)\in \R^p$ such that the relation  $ X_{a}$
 has infinitely many quasi-singular points.
Using expression  (\ref{eq:finiteQS}), the set  $A$  is $\Sx$-definable, thus it is \Ss-definable because so is   $X$. If this set were nonempty,
 by Theorem \ref{th:quantifier-elimination-for-R-plus} it would contain an element of $\Q^p$, which means that there exists a rational generalized projection of $X$ which has infinitely many quasi-singular points, a contradiction.

Conversely assume that  every {generalized
projection} of $X$ has finitely many quasi-singular points, and let us prove that  $X$ satisfies all conditions of Theorem \ref{th:crit-n}. Condition (DS) follows from the fact that every rational section of $X$ is a generalized projection of $X$ thus
  is \Ss-definable by our induction hypothesis.  For conditions (FSP) and (RSP), we observe that $X$ is a generalized projection
  of itself thus the set of $X-$quasi-singular points is finite. By Lemma \ref{le:defquasi} this set is
  $\Sx$-definable thus it is a subset of $\Q^n$ by our assumption on $\Sx$, and the result follows from Lemma \ref{le:sing-quasi-sing}.
\end{proof}

\subsection{Defining  \Ss-definability}

The formulation of conditions in Proposition \ref{pr:altern-cri-quasi} allows us to express \Ss-definability as a sentence in the structure $\Sx$ itself.

\begin{theorem}\label{th:seldef}	
	Let $n \geq 1$ and  $X\subseteq \R^{n}$ be such that  every nonempty $\Sx$-definable relation contains a point with rational components. There exists a $\Sx$-sentence $\Phi_{n}$ (which is uniform in $X$)  which holds in $\Sx$ if and only if $X$ is \Ss-definable.
\end{theorem}

\begin{proof}
{Let $[n]$ denote the set $\{1, \ldots, n\}$.}
By Proposition \ref{pr:altern-cri-quasi} it suffices to express the fact that every generalized projection of $X$ has finitely many quasi-singular points.
This  leads us to consider all possible
generalized projections obtained by freezing a subset $[n] \setminus I$ of components as in Definition \ref{de:generalized-section}.
Let   $\R^{I}$  denote the product of copies of $\R$ indexed by $I$
and for all $x,y\in \R^{I}$  set  $|x-y|_{I}= |(x+z) -(y+z)|$ for any $z\in \R^{[n]\setminus I}$.
For $x\in \R^{I}$ and $r\geq 0$   set
$B_{I}(x,r)=\{y\in \R^{I} \mid |x-y|_{I} <r\}$. We use the pair $(n,I)$ as a parameter
for  the predicates $\phi, QS,FS$ (see Lemma \ref{le:defquasi}).
{ The symbol $\xi$ stands for the  subvector with frozen components (we have $\xi \in \R^{[n]\setminus I}$). With some abuse of notation we write $\xi+z$ for $\xi \in \R^{[n]\setminus I}$ and $z \in \R^I$, that is, if $\xi=(\xi_i)_{i \in \R^{[n]\setminus I}}$
and  $z=(z_i)_{i \in I}$ then $\xi+z=(w_1,\dots,w_n)$ with $w_i=z_i$ if $i \in I$ and $w_i=\xi_i$ otherwise. We can define the predicates $\phi_{n,I}$, $QS_{n,I}$ and $FS_{n,I}$ as follows:}
\begin{equation}
\nonumber
\begin{array}{ll}
\phi_{n,I}(X,\xi,x, r,s,v)\equiv & \xi\in  \R^{[n]\setminus I} \wedge  x, v\in \R^{ I} \wedge  0< |v|_{I}< s\\
& \wedge \forall y \in  \R^{ I}   (y,y+v\in B_{I}(x,r) \rightarrow (\xi+ y\in X \leftrightarrow  \xi+ y+v \in X))
\end{array}
\end{equation}
\begin{equation}
\nonumber
\begin{array}{ll}
QS_{n,I}(x,\xi, X)\equiv &  \neg \exists r \exists s  (( 0<s<r) \wedge \exists v\ (\phi_{n,I}(X,\xi,x, r,s,v)) \\
 &\wedge \forall v\ (\phi_{n,I}(X,\xi,x,r,s,v) \rightarrow \phi_{n,I}(X,\xi,x, r,s,\frac{v}{2}) )  \\
&\wedge  \forall u \ (((|u|_{I}<s) \wedge \forall \epsilon>0 (\exists v\ (\phi_{n,I}(X,\xi,x, r,s,v)\wedge |v-u|_{I}<\epsilon )))\\
& \phantom{xxxxxxxxxxxxxxxxx}  \rightarrow  \phi_{n,I}(X,\xi,x, r,s,u)))
\end{array}
\end{equation}
\begin{equation}
\nonumber
\begin{array}{ll}
FS_{n,I}(X,\xi)\equiv &  (\exists t>0 \forall x \    (QS_{n,I}(x,\xi, X)   \rightarrow |x|_{I}<t)\\
&  \wedge (\exists  s>0
(\forall x \forall y ( (QS_{n,I}(x,\xi, X) \wedge QS_{n,I}(y,\xi, X) \wedge x\not=y)\rightarrow |x-y|_{I}>s)
\end{array}
\end{equation}

This leads to the following definition of  $\Phi_{n}$:
\begin{equation}
\label{eq:ouf}
\Phi_{n}\equiv \bigwedge_{I\subseteq [n]} \forall \xi\in  \R^{[n]\setminus I} \ \  FS_{n,I}(X,\xi).
\end{equation}

\vspace*{-7mm}
\end{proof}

{
\begin{remark}
	
	One can prove that Theorem \ref{th:seldef} does not hold anymore if we remove the assumption that  every nonempty $\Sx$-definable relation contains a point with rational components. Indeed consider $n=1$ (the case $n \geq 1$ easily reduces to this case) and a singleton set $X=\{x\} \subseteq \R$. Then
	by Theorem \ref{th:quantifier-elimination-for-R-plus}, $X$ is \Ss-definable if and only if $x \in \Q$. Thus if there exists a $\Sx$-sentence $\Phi_{n}$ which expresses that $X$ is \Ss-definable, then it is easy to transform $\Phi_{n}$ into a \Ss-formula $\Phi'_n(x)$ which defines $\Q$ in \Ss, and this contradicts Theorem \ref{th:quantifier-elimination-for-R-plus}.
\end{remark}
}

\subsection{Extensions to \Ls-definability}
\label{sec:extensions}

We extend the previous results to the case of \Ls-definability.
Here  $\Tx$ stands for
the structure \LXs{X} with  $n \geq 1$ and $X \subseteq \R^n$. We prove that if every nonempty $\Tx$-definable relation contains a point with rational components then one can express the property that $X$ is \Ls-definable with a $\Tx$-sentence.

The construction is based on the decomposition of any set of reals into ``integer'' and  ``fractional'' sets, which allows us to reduce the \Ls-definability of $X$, on  one hand to the $\langle \Z,+, < \rangle$-definability of some subsets of $\Z^n$ and on  the other hand  to the \Ss-definability of a collection of subsets of $[0,1)^n$. In order to express these two kinds of properties in $\Tx$, we rely respectively on Muchnik's Theorem \cite{Muchnik03} and on
Theorem \ref{th:seldef}.

We start with a  property which holds for all relations under no particular assumption.
Given a relation $X\subseteq \R^{n}$ consider the denumerable set of distinct restrictions of $X$ to unit hypercubes, i.e.,
$$
 \tau_{a}\big(X \cap ([a_{1}, a_{1}+1), \cdots, [a_{n}, a_{n}+1) ) \big)
$$
where  $a=(a_{1},   \cdots, a_{n})\in \Z^{n}$ and $\tau_{a}$ is the translation $x\mapsto x-a$. Let $\Delta_{m}$ denote this
collection of  sets  where $m$ runs over some denumerable set $M$. For each $m \in M$,  let $\Sigma_{m}\subseteq \Z^{n}$ satisfy the condition
$$
x\in \Sigma_{m} \leftrightarrow x + \Delta_{m} = X \cap  \big( [ x_{1},  x_{1},+1), \cdots, [ x_{n},  x_{n},+1) \big)
$$
Observe that  the   decomposition
\begin{equation}
\label{eq:unique-decomposition}
X= \bigcup_{m \in M} \Sigma_{m} + \Delta_{m}
\end{equation}
is unique by construction.
In the particular case of \Ls-definable relations we have the following result.

\begin{proposition}{(\cite[Theorem 7]{BFL08}, see also \cite{FV59})}
\label{prop-decomp}
A relation $X \subseteq \R^n$ is \Ls-definable if and only if  in the decomposition (\ref{eq:unique-decomposition})
the following three conditions hold:
\begin{description}
\item  {\em (FU)} the set $M$ in (\ref{eq:unique-decomposition}) is finite.
\item  {\em  (IP)}  each $\Sigma_{m}$ is $\langle \Z,+, < \rangle$-definable.
\item  {\em  (FP)}  each $\Delta_{m}$ is \Ss-definable.
\end{description}
\end{proposition}

\begin{proposition}\label{pr:reduc-presb}
For all $n \geq 1$ and $X \subseteq \R^n$, if every nonempty $\Tx$-definable relation contains a point with rational components then there exists a $\Tx$-sentence $\Gamma_n$ (uniform in $X$) which holds if and only if $X$ is \Ls-definable
\end{proposition}

\begin{proof}
In view of Proposition \ref{prop-decomp}  it suffices to show that the three conditions are expressible in
\str{\R,+,<,\Z,X}.
Let $\Phi(X)$ be the $\Tx$-formula which states that $X$ is  \Ss-definable, see equation (\ref{eq:ouf}).

\nl Condition (FU): Let   $x \approx y$ denote  the equivalence relation which says that the two points $x$ and $y$  belong to the same $\Sigma_{m}$ in the decomposition \ref{eq:unique-decomposition}.
It is expressed by the  $\Tx$-formula
\begin{equation}
\label{eq:approx}
x \in \Z^{n} \wedge y\in \Z^{n} \wedge  \forall  z \in [0,1)^n\ (x+z \in X \leftrightarrow y+z \in X)
\end{equation}
 The finiteness of the number of classes   is expressed  by the $\Tx$-formula
$$
\exists N \ \forall  x \in \Z^{n} \  \exists y \in \Z^{n} \ ( |y| < N \wedge y\approx x)
$$
\nl Condition (IP):  By \cite[Thm 1]{Muchnik03} for every $Y \subseteq \Z^n$ there exists a \str{\Z,+,<,Y}-formula $\Psi(Y)$ (uniform in $Y$) which holds if and only if $Y$ is \str{\Z,+,<}-definable. Let $\Psi^*(Y)$ denote the \str{\R,\Z,+,<,Y}-formula obtained from $\Psi(Y)$ by relativizing all quantifiers to $\Z$. Given $Y \subseteq \Z^n$ (seen as a subset of $\R^n$), the formula $\Psi^*(Y)$ holds in \str{\R,\Z,+,<,Y} if and only if $Y$ is \str{\Z,+,<}-definable. Thus we can express  in $\Tx$  the fact that all $\approx$-equivalence classes are \str{\Z,+,<}-definable with the formula
$$
\forall  x\in \Z^{n}  \Psi^*((y\in \Z^{n} \wedge y\approx x))
$$
\nl Condition (FP):   The fact that every hypercube of $X$   of unit side is \Ss-definable is expressed by\vspace*{-1.8mm}
$$
\forall x_{1}, \ldots x_{n}  \in \Z^{n} \  \Phi((0\leq y_{1}< 1\wedge \cdots \wedge 0\leq y_{n}< 1 \wedge (x_{1}+ y_{1}, \cdots, x_{n}+y_{n}) \in X)).
$$

\vspace*{-8mm}
\end{proof}

\section{Application to decidability}
\label{sec:applications}

\subsection{Deciding \Ss-definability and \Ls-definability}
\label{subsec:decid-general}

Theorem \ref{th:seldef} and Proposition  \ref{pr:reduc-presb} prove the existence of definable criteria for \Ss-definability and \Ls-definability for a given relation $X$, respectively. If $X$ is definable in some decidable expansion of \Ss\ (resp. \Ls)  then we can obtain effective criteria. This can be formulated as follows.

\begin{theorem}\label{th:eff1}
	Let $\+M$ be any decidable expansion of \Ss\  such that every nonempty  $\+M$-definable relation contains a point with rational components. Then it is decidable whether a  $\+M$-definable relation $X \subseteq \R^n$ is \Ss-definable.
\end{theorem}

\begin{proof}
	Assume that $X$ is $\+M$-definable by the formula $\psi(x)$. In Equation (\ref{eq:ouf}), if we substitute $\psi(x)$ for every occurrence
	of $x\in X$
	then we obtain a  $\+M$-sentence  which holds if and only if $X$
	is \Ss-definable, and the result follows from the decidability of $\+M$.
\end{proof}

\begin{theorem}\label{th:eff2}
	Let $\+N$ be any decidable expansion of \Ls\ such that every nonempty  $\+N$-definable relation contains a point with rational components. Then it is decidable whether a  $\+N$-definable relation $X \subseteq \R^n$ is \Ls-definable (resp. whether a  $\+N$-definable relation $X \subseteq \R^n$ is \Ss-definable).
\end{theorem}

\begin{proof}
	The  claim about \Ss-definability follows immediately from the fact that $\+N$ satisfies the conditions of Theorem \ref{th:eff1}. For \Ls-definability, we use the same idea as for the proof of Theorem \ref{th:eff1}, but instead of $\Phi_n$ we use the sentence $\Gamma_n$ of Proposition \ref{pr:reduc-presb}.	
\end{proof}

\subsection{Application to recognizable numerical relations}

We finally apply the results of Section \ref{subsec:decid-general} to the class of $k$-recognizable relations on reals.

Let us recall that given an  integer {base} $k \geq 2$ and a non-negative real $x$, a {\em $k$-encoding} of $x$ is any right  infinite word on the alphabet
$\Sigma_k= \{0, \ldots, k-1\}\cup \{\star\}$ of the form $w=a_{p}\ldots a_{1}\star  a_{0} a_{-1} a_{-2 } \ldots$ such that $a_i \in \{0, \ldots, k-1\}$ for every $i \leq p$ and $x= \sum_{i \leq p} a_i k^i$. The definition extends to the case of negative reals $x$ by using the $k$'s complement representation method where the leftmost digit equals $k-1$: a {$k-$encoding} of $x$ is a right  infinite word on the alphabet $\Sigma_k$ of the form $w=a_{p}\ldots a_{1}\star  a_{0} a_{-1} a_{-2 } \ldots$ where $a_p=k-1$ and $x=-k^p +  \sum_{i \leq p-1} a_i k^i$. Note that every real has infinitely many $k$-encodings.

In order for an automaton to be able to process $n$-tuples of representations in base $k$ of reals we
prefix it, if necessary,  with as few occurrences of  $0$ for the nonnegative components or $k-1$ to the negative components
so that the $n$ components have the same length to the left of the symbol $\star$. This does  not change the numerical values represented.
By identifying a real with its $k$-encodings, relations of arity $n$ on $\R$ can thus be viewed  as subsets of $n$-tuples
of sequences
on $\{0, \ldots, k-1\}\cup \{\star\}$, i.e., as  subsets of
$$
(\{0, \ldots, k-1\}^{n})^{*} \{\overbrace{(\star, \ldots, \star)}^{n \text{\tiny times}}\} (\{0, \ldots, k-1\}^{n})^{\omega}
$$

\begin{definition}
	\label{de:recognizable}
	A relation $X\subseteq \R^{n}$ is $k$-\emph{recognizable}
	if the set of $k$-encodings of its elements
	is recognized by some deterministic Muller-automaton.
\end{definition}

The collection of  recognizable relations has a natural logical characterization.

\begin{theorem}\cite[Thm 5 and 6]{BRW1998}
	\label{th:brw98}
	Let $k \geq 2$ be an integer. A subset of $\R^{n}$ is $k$-recognizable if and only if it is definable in
	$\langle \R,\Z,+,<,X_k\rangle$ where $X_k \subseteq \R^3$ is such that $X_k(x,y,z)$ holds if and only if $y$ is a power of $k$ and $z$ is the coefficient of
	$y$ in some $k-$encoding of $x$.
	
	Consequently, since the emptiness problem for recognizable relations is decidable, the theory of $\langle \R,\Z,+,<,X_k \rangle$ is decidable.
\end{theorem}

Moreover the class of \Ls-definable relations enjoys the following characterization.

\begin{theorem}\cite{BB2009,BBB2010,BBL09}
	\label{th:all-bases}
	A subset of $\R^{n}$ is \Ls-definable if and only if it is $k$-recognizable for every integer $k \geq 2$.
\end{theorem}

As a consequence, deciding whether a $k$-recognizable relation is $l-$recognizable for every base $l \geq 2$ amounts to decide whether it is \Ls-definable. We can prove the following result.

\begin{theorem}\label{th:rec-z-r}
	Given an integer $k \geq 2$, it is decidable whether a $k$-recognizable relation $X \subseteq \R^n$   is \Ls-definable (resp. whether a $k-$recognizable relation $X \subseteq \R^n$   is \Ss-definable).
\end{theorem}

\begin{proof} By Theorems \ref{th:eff2} and \ref{th:brw98} it suffices to prove that every nonempty $k$-recognizable relation $Y \subseteq  \R^n$ contains an element in $\Q^{n}$. By our assumption the set of $k$-encodings of elements of $Y$ is nonempty and is recognized by a finite Muller automaton, thus it contains an ultimately periodic $\omega-$word, which is the $k-$encoding of some element of $\Q^n$.
\end{proof}

\subsection*{Acknowledgment}

We wish to thank the anonymous referees and Fran\c{c}oise Point for their useful comments and suggestions.

\end{document}